\documentclass[12pt]{amsart}
\usepackage{amssymb}
\newcommand{\deleted}[1]{}
\newcommand{\delete}[1]{}
\newcommand{\mynotes}[1]{}
\newcommand\notes[1]{}
\usepackage{amsfonts}
\usepackage{stmaryrd}
\usepackage{}
\usepackage{amsfonts}      
\usepackage{txfonts}
\usepackage{amssymb}
\usepackage{eucal}
\usepackage{graphicx}
\usepackage{amsmath}
\usepackage{amscd}
\usepackage[all]{xy}           
\usepackage[active]{srcltx} 

\usepackage{amsfonts,latexsym}
\usepackage{xspace}
\usepackage{epsfig}
\usepackage{float}
\usepackage{color}
\usepackage{fancybox}
\usepackage{colordvi}
\usepackage{multicol}
\usepackage{colordvi}
\usepackage{stmaryrd}
\usepackage[normalem]{ulem}
\usepackage[colorlinks,final,backref=page,hyperindex,hypertex]{hyperref}
\usepackage[hypertex]{hyperref} 

\newcommand\changed[1]{#1}
\changed{}

\newtheorem{theorem}{Theorem}[section]
\newtheorem{lemma}[theorem]{Lemma}
\newtheorem{coro}[theorem]{Corollary}

\theoremstyle{definition}

\newtheorem{remark}[theorem]{Remark}

\deleted{\newtheorem{theorem}{Theorem}[section]
\newtheorem{prop-def}{Proposition-Definition}[section]
\newtheorem{coro-def}{Corollary-Definition}[section]

}

\newcommand{\nc}{\newcommand}
\nc{\tred}[1]{\textcolor{red}{#1}} \nc{\tblue}[1]{\textcolor{blue}{#1}} \nc{\tgreen}[1]{\textcolor{green}{#1}} \nc{\tpurple}[1]{\textcolor{purple}{#1}} \nc{\btred}[1]{\textcolor{red}{\bf #1}} \nc{\btblue}[1]{\textcolor{blue}{\bf #1}} \nc{\btgreen}[1]{\textcolor{green}{\bf #1}} \nc{\btpurple}[1]{\textcolor{purple}{\bf #1}}

\renewcommand{\Bbb}{\mathbb}

\newcommand{\efootnote}[1]{}

\renewcommand{\textbf}[1]{}

\nc{\mlabel}[1]{\label{#1}}  
\nc{\mcite}[1]{\cite{#1}}  
\nc{\mref}[1]{\ref{#1}}  
\nc{\mbibitem}[1]{\bibitem{#1}} 

\delete{
\nc{\mlabel}[1]{\label{#1}  
{\hfill \hspace{1cm}{\bf{{\ }\hfill(#1)}}}}
\nc{\mcite}[1]{\cite{#1}{{\bf{{\ }(#1)}}}}  
\nc{\mref}[1]{\ref{#1}{{\bf{{\ }(#1)}}}}  
\nc{\mbibitem}[1]{\bibitem[\bf #1]{#1}} 
}

\renewcommand\geq{\geqslant}
\renewcommand\leq{\leqslant}

\renewcommand\bar[1]{\overline{#1}}
\renewcommand\tilde[1]{\widetilde{#1}}

\nc{\rbw}{\mathfrak{R}} \nc{\brp}{\mathrm{brp}} \nc{\lead}{\mathrm{Lead}} \nc{\Id}{\mathrm{Id}} \nc{\Irr}{\mathrm{Irr}} \nc{\vx}{\sigma} \nc{\vy}{\tau} \nc{\dvx}{\sigma^{(1)}} \nc{\dvy}{\tau^{(1)}} \nc{\done}{\vep} \nc{\citep}[1]{\cite{#1}} \nc{\wt}{\mathrm{wt}} \nc{\bre}[1]{|#1|} \nc{\mapmonoid}{\frakM} \nc{\disjoint}{\frakM'}
\nc{\ncpoly}[1]{\langle #1\rangle}  
\nc{\mapm}[1]{\lfloor\!|{#1}|\!\rfloor}
\nc{\diff}[1]{{}^\NC\{ #1 \}} \nc{\disj}[1]{\{{#1}\}'} \nc{\mdisj}[1]{\frakM'(#1)} \nc{\brho}{\bar{\rho}} \nc{\om}{\bar{\frakm}} \nc{\frakn}{\mathfrak n} \nc{\ddeg}[1]{^{(#1)}} \nc{\opset}{X} \nc{\genset}{{Z}} \nc{\NC}{\mathrm{{NC}}} \nc{\leaf}{\mathrm{leaf}} \nc{\twig}{\mathrm{twig}} \nc{\fe}{\mathrm{fl}} \nc{\munderline}[1]{#1} \nc{\bo}{o} \nc{\dep}{\mathrm{depth}} \nc{\ofe}{\mathrm{ofl}} \nc{\dfe}{\mathrm{dfe}} \nc{\fex}{\mathrm{fex}} \nc{\dl}{\mathrm{dlex}} \nc{\db}{\mathrm{db}} \nc{\lex}{\mathrm{lex}} \nc{\clex}{\mathrm{clex}} \nc{\dgp}{\mathrm{dgp}} \nc{\dgx}{\mathrm{dgx}} \nc{\br}{\mathrm{br}} \nc{\obd}{\mathrm{odb}} \nc{\ob}{\mathrm{ob}}
\nc{\loc}{location\xspace}
\nc{\occ}{occurrence\xspace}
\nc{\occs}{occurrences\xspace}
\nc{\pla}{placement\xspace}
\nc{\plas}{placements\xspace}

\nc{\bin}[2]{ (_{\stackrel{\scs{#1}}{\scs{#2}}})}  
\nc{\binc}[2]{ \left (\!\! \begin{array}{c} \scs{#1}\\
    \scs{#2} \end{array}\!\! \right )}  
\nc{\bincc}[2]{  \left ( {\scs{#1} \atop
    \vspace{-1cm}\scs{#2}} \right )}  
\nc{\bs}{\bar{S}} \nc{\cosum}{\sqsubset} \nc{\la}{\longrightarrow} \nc{\rar}{\rightarrow} \nc{\dar}{\downarrow} \nc{\dprod}{**} \nc{\dap}[1]{\downarrow \rlap{$\scriptstyle{#1}$}} \nc{\md}{\mathrm{dth}} \nc{\uap}[1]{\uparrow \rlap{$\scriptstyle{#1}$}} \nc{\defeq}{\stackrel{\rm def}{=}} \nc{\disp}[1]{\displaystyle{#1}} \nc{\dotcup}{\ \displaystyle{\bigcup^\bullet}\ } \nc{\gzeta}{\bar{\zeta}} \nc{\hcm}{\ \hat{,}\ } \nc{\hts}{\hat{\otimes}} \nc{\barot}{{\otimes}} \nc{\free}[1]{\bar{#1}} \nc{\uni}[1]{\tilde{#1}} \nc{\hcirc}{\hat{\circ}} \nc{\leng}{\ell} \nc{\lleft}{[} \nc{\lright}{]} \nc{\lc}{\lfloor} \nc{\rc}{\rfloor}
\nc{\lb}{[} 
\nc{\rb}{]} 
\nc{\curlyl}{\left \{ \begin{array}{c} {} \\ {} \end{array}
    \right.  \!\!\!\!\!\!\!}
\nc{\curlyr}{ \!\!\!\!\!\!\!
    \left. \begin{array}{c} {} \\ {} \end{array}
    \right \} }
\nc{\longmid}{\left | \begin{array}{c} {} \\ {} \end{array}
    \right. \!\!\!\!\!\!\!}
\nc{\onetree}{\bullet} \nc{\ora}[1]{\stackrel{#1}{\rar}}
\nc{\ola}[1]{\stackrel{#1}{\la}}
\nc{\ot}{\otimes} \nc{\mot}{{{\boxtimes\,}}} \nc{\otm}{\overline{\boxtimes}} \nc{\sprod}{\bullet} \nc{\scs}[1]{\scriptstyle{#1}} \nc{\mrm}[1]{{\rm #1}} \nc{\msum}{\sum\limits}
\nc{\margin}[1]{\marginpar{\rm #1}}   
\nc{\dirlim}{\displaystyle{\lim_{\longrightarrow}}\,} \nc{\invlim}{\displaystyle{\lim_{\longleftarrow}}\,} \nc{\mvp}{\vspace{0.3cm}} \nc{\tk}{^{(k)}} \nc{\tp}{^\prime} \nc{\ttp}{^{\prime\prime}} \nc{\svp}{\vspace{2cm}} \nc{\vp}{\vspace{8cm}} \nc{\proofbegin}{\noindent{\bf Proof: }}
\nc{\proofend}{$\blacksquare$ \vspace{0.3cm}}
\nc{\modg}[1]{\!<\!\!{#1}\!\!>}
\nc{\intg}[1]{F_C(#1)} \nc{\lmodg}{\!<\!\!} \nc{\rmodg}{\!\!>\!} \nc{\cpi}{\widehat{\Pi}}
\nc{\sha}{{\mbox{\cyr X}}}  
\nc{\shap}{{\mbox{\cyrs X}}} 
\nc{\shpr}{\diamond}    
\nc{\shp}{\ast} \nc{\shplus}{\shpr^+}
\nc{\shprc}{\shpr_c}    
\nc{\msh}{\ast} \nc{\zprod}{m_0} \nc{\oprod}{m_1} \nc{\vep}{\varepsilon} \nc{\labs}{\mid\!} \nc{\rabs}{\!\mid}
\nc{\astarrow}{\overset{\raisebox{-3pt}{$\ast$}}{\rightarrow}}
\nc{\lastarrow}{\overset{\raisebox{-3pt}{$\ast$}}{\leftarrow}}
\nc{\mpu}{u^{\ast}}
\nc{\mpv}{v^{\ast}}
\nc{\mpw}{w^{\ast}}
\nc{\mpx}{x^{\ast}}
\nc{\dps}{\dotplus}

\nc{\dth}{d} \nc{\mmbox}[1]{\mbox{\ #1\ }} \nc{\fp}{\mrm{FP}} \nc{\rchar}{\mrm{char}} \nc{\Fil}{\mrm{Fil}} \nc{\Mor}{Mor\xspace} \nc{\gmzvs}{gMZV\xspace} \nc{\gmzv}{gMZV\xspace} \nc{\mzv}{MZV\xspace} \nc{\mzvs}{MZVs\xspace} \nc{\Hom}{\mrm{Hom}} \nc{\id}{\mrm{id}} \nc{\im}{\mrm{im}} \nc{\incl}{\mrm{incl}} \nc{\map}{\mrm{Map}} \nc{\mchar}{\rm char} \nc{\nz}{\rm NZ} \nc{\supp}{\mathrm Supp}
\nc{\mo}{\mathbf o}
\nc{\pl}{\mathfrak{p}}

\nc{\Alg}{\mathbf{Alg}} \nc{\Bax}{\mathbf{Bax}} \nc{\bff}{\mathbf f} \nc{\bfk}{{\bf k}} \nc{\bfone}{{\bf 1}} \nc{\bfx}{\mathbf x} \nc{\bfy}{\mathbf y}
\nc{\base}[1]{\bfone^{\otimes ({#1}+1)}} 
\nc{\Cat}{\mathbf{Cat}} \delete{}
\nc{\detail}{\marginpar{\bf More detail}
    \noindent{\bf Need more detail!}
    \svp}
\nc{\Int}{\mathbf{Int}} \nc{\Mon}{\mathbf{Mon}}
\nc{\rbtm}{{shuffle }} \nc{\rbto}{{Rota-Baxter }} \nc{\remarks}{\noindent{\bf Remarks: }} \nc{\Rings}{\mathbf{Rings}} \nc{\Sets}{\mathbf{Sets}}

\nc{\BA}{{\Bbb A}} \nc{\CC}{{\Bbb C}} \nc{\DD}{{\Bbb D}} \nc{\EE}{{\Bbb E}} \nc{\FF}{{\Bbb F}} \nc{\GG}{{\Bbb G}} \nc{\HH}{{\Bbb H}} \nc{\LL}{{\Bbb L}} \nc{\NN}{{\Bbb N}} \nc{\KK}{{\Bbb K}} \nc{\QQ}{{\Bbb Q}} \nc{\RR}{{\Bbb R}} \nc{\TT}{{\Bbb T}} \nc{\VV}{{\Bbb V}} \nc{\ZZ}{{\Bbb Z}}

\nc{\cala}{{\mathcal A}} \nc{\calb}{{\mathcal B}}
\nc{\calc}{{\mathcal C}} \nc{\cald}{{\mathcal D}} \nc{\cale}{{\mathcal E}} \nc{\calf}{{\mathcal F}} \nc{\calg}{{\mathcal G}} \nc{\calh}{{\mathcal H}} \nc{\cali}{{\mathcal I}} \nc{\call}{{\mathcal L}} \nc{\calm}{{\mathcal M}} \nc{\caln}{{\mathcal N}} \nc{\calo}{{\mathcal O}} \nc{\calp}{{\mathcal P}} \nc{\calr}{{\mathcal R}} \nc{\cals}{{\mathcal S}} \nc{\calt}{{\mathcal T}} \nc{\calw}{{\mathcal W}} \nc{\calk}{{\mathcal K}} \nc{\calx}{{\mathcal X}} \nc{\CA}{\mathcal{A}}

\nc{\fraka}{{\mathfrak a}} \nc{\frakA}{{\mathfrak A}} \nc{\frakb}{{\mathfrak b}} \nc{\frakB}{{\mathfrak B}} \nc{\frakZ}{{\mathfrak Z}}\nc{\frakD}{{\mathfrak D}} \nc{\frakH}{{\mathfrak H}} \nc{\frakM}{{\mathfrak M}} \nc{\bfrakM}{\overline{\frakM}} \nc{\frakm}{{\mathfrak m}} \nc{\frakP}{{\mathfrak P}} \nc{\frakN}{{\mathfrak N}} \nc{\frakp}{{\mathfrak p}} \nc{\frakS}{{\mathfrak S}} \nc{\frakx}{{\mathfrak x}} \nc{\ox}{\bar{\frakx}} \nc{\frakX}{{\mathfrak X}} \nc{\fraky}{{\mathfrak y}} \nc\dop{\delta}
\nc{\Reduce}{{\rm Red}}

\nc{\bfs}{{\bf s}}
\font\cyr=wncyr10 \font\cyrs=wncyr7
\nc{\redt}[1]{\textcolor{red}{#1}}

\nc{\li}[1]{\textcolor{red}{\tt Li:#1}} \nc{\lp}[1]{\textcolor{blue}{\tt Lei:#1}} \nc{\mb}[1]{\textcolor{purple}{\tt Ma:#1}}

\begin{document}
\title[Restricted decomposition formulas of MZVs]{Applications of shuffle product to restricted decomposition formulas for multiple zeta values}

\author{Li Guo}
\address{
    Department of Mathematics and Computer Science,
         Rutgers University,
         Newark, NJ 07102, USA}
\email{liguo@rutgers.edu}

\author{Peng Lei}
\address{Department of Mathematics,
    Lanzhou University,
    Lanzhou, Gansu 730000, China}
\email{leip@lzu.edu.cn}

\author{Biao Ma}
\address{Cuiying Honors College and Department of Mathematics,
    Lanzhou University,
    Lanzhou, Gansu 730000, China}
\email{mab2010@lzu.edu.cn}

\hyphenpenalty=8000
\date{\today}

\begin{abstract}
In this paper we obtain a recursive formula for the shuffle product and apply it to derive two restricted decomposition formulas for multiple zeta values (MZVs).
The first formula generalizes the decomposition formula of Euler and is similar to the restricted formula of Eie and Wei for MZVs with one strings of 1's. The second formula generalizes the previous results to the product of two MZVs with one and two strings of 1's respectively.
\end{abstract}

\delete{
\begin{keyword}
Multiple zeta values, Euler's decomposition formula, shuffle product, restricted decomposition formula
\end{keyword}
}
\maketitle

\tableofcontents

\hyphenpenalty=8000 \setcounter{section}{0}


\section{Introduction}
A {\bf multiple zeta value (MZV)} is the special value of the complex valued function
$$\zeta(s_1,\cdots,s_k)=\sum_{n_1>\cdots >n_k\geq 1} \frac{1}{n_1^{s_1}\cdots n_k^{s_k}}$$
at positive integers $s_1,\cdots,s_k$ with $s_1\geq 2$ to ensure the convergence of the nested sum. MZVs are natural generalizations of the Riemann zeta values $\zeta(s)$ to multiple variables.
They were introduced in the 1990s with motivations from number theory~\mcite{Za}, combinatorics~\mcite{Ho0} and quantum field theory~\mcite{BK}. Since then the subject has turned into an active area of research that involves many areas of mathematics and mathematical physics~\mcite{Ca1}. Its number theoretic significance can be seen from the recent theorem of Brown~\mcite{Brn,Za3} that all periods of mixed Tate motives unramified over $\ZZ$ are $\QQ[\frac{1}{2\pi i}]$-linear combinations of MZVs (See also \mcite{Go,GM,Te}).
\smallskip

A major goal on MZVs is to determine all algebraic relations among MZVs. In the two variable case, this problem has been studied over two hundred years ago by Goldbach and Euler ~\mcite{Eu,Sa}.
Among Euler's major discoveries are his {\bf
sum formula} $$ \sum_{i=2}^{n-1} \zeta(i,n-i)=\zeta(n)$$ expressing
one Riemann zeta values as a sum of double zeta values and the
{\bf decomposition formula}
\begin{equation}
\zeta(r) \zeta(s)
 = \sum_{k=0}^{s-1} \binc{r+k-1}{k} \zeta(r+k,s-k)
+ \sum_{k=0}^{r-1} \binc{s+k-1}{k} \zeta(s+k,r-k), \quad r,s\geq 2,
\mlabel{eq:euler}
\end{equation}
expressing the product of two Riemann zeta values as a sum of double zeta values.

Soon after MZVs were introduced, Euler's sum formula was generalized to MZVs~\mcite{Ho0,Gr,Za2} as the well-known sum formula, followed by quite a few other generalizations~\mcite{Br,ELO,GX2,GZ,HO,KTY,Oh,OO,OZa,OZ}.
Generalizations of Euler's decomposition formula to MZVs came much slowly even though Euler's formula had been revisited and found applications in modular forms~\mcite{BBG,3BL2,GKZ}.
Euler's decomposition formula was generalized in~\mcite{Br2,Zh} to the product of two $q$-zeta values. More recently, Euler's decomposition formula for multi-variable MZVs were obtained in~\mcite{EW,GX3}. In~\mcite{GX3}, Euler's decomposition formula is generalized to the product of any two MZVs by the algebraic method of double shuffle product. By an analytic method, Euler's decomposition formula is generalized in~\mcite{EW} to the product of two MZVs of the form
$\zeta(m,\{1\}^k)$ with one string of 1's: $\{1\}^k:=\underbrace{1,\cdots,1}_{k \text{ terms }}$. A formula for MZVs with strings of 1's is often called a restricted decomposition formula (see Section~\mref{ss:general} for a general discussion on restricted decomposition formulas).

In this paper, we give a new recursive formula (Theorem~\mref{thm:rec}) for the shuffle product and apply it to obtain two restricted decomposition formulas. One is for MZVs with one strings of 1's similar to the one in~\mcite{EW}. The other one is for MZVs with one and two strings of 1's. The following are the two theorems on MZVs. We will use the convention that, if $i=0$, then the string $\alpha_1+m,\cdots,\alpha_{i},\alpha_{{i}+1}$ is taken to be $\alpha_1+m$.

\begin{theorem}For positive integers~$m,n,j$ and $k$,~we have
\allowdisplaybreaks{\begin{eqnarray}
&&\zeta(m+1,{\{1\}}^{j-1})\zeta(n+1,{\{1\}}^{k-1})\notag\\
&=&\sum_{
\substack{|\alpha|=n-n_1+j_1+1\\
j_1+j_2=j,j_i\geq 0
\\ 0\leq n_1 \leq n}} \bincc{m-1+n_1}{m-1}\bincc{j_2+k-1}{k-1}\zeta(\alpha_1+m+n_1,\alpha_2,\cdots,\alpha_{j_1},\alpha_{j_1+1},{\{1\}}^{j_2+k-1})\notag\\
&&+\sum_{
\substack{|\beta|=m_2+k-t+1\\
m_1+m_2=m-1,m_i\geq 0
\\ 0\leq t \leq k-1}} \bincc{m_1+n-1}{n-1}\bincc{j+t}{j}\zeta(\beta_1+m_1+n,\beta_2,\cdots,\beta_{k-t},\beta_{k-t+1}+1,{\{1\}}^{j+t-1}),\notag
\end{eqnarray}}
where~$ |\alpha|:=\alpha_1+\alpha_2+\cdots+\alpha_{j_1}+\alpha_{j_1+1}$ with $\alpha_i\geq 1$ and $|\beta|:=\beta_1+\beta_2+\cdots+\beta_{k-t}+\beta_{k-t+1}$ with $\beta_i\geq 1.$
\mlabel{thm:01.01}
\end{theorem}
The appearance of the formula is different from the one in~\mcite{EW}. They should agree with each other after rearrangement of terms since both formulas are based the shuffle product.

When $j=k=1$, we derive Euler's decomposition formula (see Section~\mref{ss:co:01.01}).
\begin{coro}
For positive integers~$m,n$, we get
\begin{equation}
\zeta(m+1)\zeta(n+1)=\sum_{j=1}^{n+1}\bincc{m+n-j+1}{m}\zeta(m+n+2-j,j)+\sum_{j=1}^{
m+1}
\bincc{m+n-j+1}{n}\zeta(m+n+2-j,j).
\mlabel{eq:co:01.01}
\end{equation}\mlabel{col:01.01}
\end{coro}

We further obtain the following formula for the product of two MZVs one of which has two strings of 1's.

\begin{theorem}For positive integers~$m,n,j,k,s$ and $t$,~we have
\allowdisplaybreaks{
{\small
\begin{eqnarray}
&&\zeta(m+1,{\{1\}}^{j-1})\zeta(n+1,{\{1\}}^{k-1},s+1, {\{1\}}^{t-1})\notag\\
&&\notag\\
&=&\hspace{-.8cm}\sum_{
\substack{|\alpha|=n-n_1+j_1+1\\
|\tilde{\alpha}|=s+j_3
\\ j_1+j_2+j_3+j_4=j
\\ 0\leq n_1 \leq n}} \hspace{-.7cm} \bincc{m-1+n_1}{m-1}\bincc{j_2+k-1}{k-1}\bincc{j_4+t-1}{t-1}\zeta(\alpha_1+m+n_1,\alpha_2,\cdots,\alpha_{j_1},\alpha_{j_1+1},{\{1\}}^{j_2+k-1},\tilde{\alpha}_1+1,\cdots,\tilde{\alpha}_{j_3},\tilde{\alpha}_{j_3+1},{\{1\}}^{j_4+t-1})\notag\\
&&+\sum_{
\substack{ 1\leq k_1 \leq k\\ 0\leq m_1\leq m-1}}\bincc{m_1+n-1}{n-1}\sum_{
\substack{|\beta|=m-m_1+k_1,|\tilde{\beta}|=s+j_2
\\ j_1+j_2+j_3=j}} \bincc{j_1+k-k_1}{k-k_1}\bincc{j_3+t-1}{t-1}\times\notag\\
&&\qquad \zeta(\beta_1+m_1+n,\beta_2,\cdots,\beta_{k_1},\beta_{k_1+1}+1,{\{1\}}^{j_1+k-k_1-1},\tilde{\beta}_1+1,\cdots,\tilde{\beta}_{j_2},\tilde{\beta}_{j_2+1},{\{1\}}^{j_3+t-1})\notag
\\
&&\notag\\
&&+\hspace{-.8cm}\sum_{
\substack{ 1\leq s_1 \leq s\\m_1+m_2+m_3=m-1}} \hspace{-.8cm} \bincc{m_1+n-1}{n-1}\bincc{m_3+s_1-1}{s_1-1}\sum_{
\substack{|\gamma|=m_2+k+1,|\tilde{\gamma}|=s-s_1+j_1+1
\\ j_1+j_2=j}} \bincc{j_2+t-1}{t-1}\times\notag\\
&&\qquad \zeta(\gamma_1+m_1+n,\gamma_2,\cdots,\gamma_{k},\gamma_{k+1}+\tilde{\gamma}_1+m_3+s_1,\tilde{\gamma}_2,\cdots,\tilde{\gamma}_{j_1},\tilde{\gamma}_{j_1+1},{\{1\}}^{j_2+t-1})\notag\\
&&\notag \\
&&+\hspace{-.8cm}\sum_{
\substack{ m_1+m_2+m_3+m_4=m-1
\\ 1\leq t_1 \leq t}} \hspace{-.8cm}\bincc{m_1+n-1}{n-1}\bincc{m_3+s-1}{s-1}\bincc{j+t-t_1}{j}\times\notag\\
&&\sum_{
\substack{|\delta|=m_2+k\\
|\tilde{\delta}|=m_4+t_1+1
}}
\zeta(\delta_1+m_1+n,\delta_2,\cdots,\delta_{k},m_3+s+\tilde{\delta}_1,\tilde{\delta}_2,\cdots,\tilde{\delta}_{t_1},\tilde{\delta}_{t_1+1}+1,{\{1\}}^{j+t-t_1-1}),\notag \end{eqnarray}
}}
where $j_i,m_i\geq 0$ and
\allowdisplaybreaks{\begin{eqnarray}
&&|\alpha|=\alpha_1+\cdots+\alpha_{j_1+1},|\tilde{\alpha}| =\tilde{\alpha}_1+\cdots+\tilde{\alpha}_{j_3+1} \text{ with }\alpha_i, \tilde{\alpha}_i\geq 1,\notag\\
&&|\beta|=\beta_1+\cdots+\beta_{{k_1}+1},|\tilde{\beta}| =\tilde{\beta}_1+\cdots+\tilde{\beta}_{{j_2}+1} \text{ with }\beta_i, \tilde{\beta}_i\geq 1,\notag\\
&&|\gamma|=\gamma_1+\cdots+\gamma_{k+1},|\tilde{\gamma}| =\tilde{\gamma}_1+\cdots+\tilde{\gamma}_{j_1+1} \text{ with }\gamma_i, \tilde{\gamma}_i\geq 1,\notag\\
&&|\delta|=\delta_1+\cdots+\delta_{k},|\tilde{\delta}| =\tilde{\delta}_1+\cdots+\tilde{\delta}_{{t_1}+1} \text{ with }\delta_i, \tilde{\delta}_i\geq 1.\notag
\end{eqnarray}}
\mlabel{thm:01.0101}
\end{theorem}

The method in this paper can be applied to give restricted decomposition formulas for MZVs with more stings of 1's, but the complexity of the formula increases quickly as the number of strings of 1's increases. See Remark~\mref{rk:0101.0101} for the case of the product of two MZVs both with two strings of 1's.

To prove the theorems, we first obtain a recursive formula for the shuffle product in Section~\mref{sec:rec}. The recursive formula is for words of special forms discussed in this paper and is different from the usual recursive formula for the shuffle product of arbitrary words. Then our recursive formulas are applied to obtain Theorem~\mref{thm:01.01} in Section~\mref{ss:01.01} and Theorem~\mref{thm:01.0101} in Section~\mref{ss:01.0101} respectively.

\section{A recursive formula for the shuffle product}
\mlabel{sec:rec}

Let $X$ be a nonempty set. Let $S(X)$ be the free semigroup on $X$. The shuffle product algebra on $X$, denoted by $\calh^{\shap}=\calh^{\shap}_X$, is the vector space $\QQ S(X)$ spanned by $S(X)$ together with the shuffle product $\shap$.
There are several equivalent descriptions of the shuffle product. We will use the following description in terms of order preserving maps. For details see~\cite{GX3} or \cite[\S~3.1.4]{Gub} for example.

Denote $[k]:=\{1,\cdots,k\}$ and $[k,\ell]=[k,\cdots, \ell]$ for $1\leq k<\ell$. For $m, n\geq 1$, denote
$$J_{m,n}:=\bigg\{ (\varphi, \psi)\,\bigg|\, \begin{array}{l} \varphi:[m]\to [m+n] \text{ and } \psi:[n]\to [m+n] \text{ are order preserving }\\ \text{and injective maps such that } \im\,\varphi \cup \im\, \psi = [m+n]\end{array} \bigg\}.$$
Note that the conditions for $(\varphi,\psi)$ implies that $\im\,\varphi \cap \im\,\psi=\emptyset$. Thus for each $i\in [m+n]$, either $\varphi(j)=i$ for some $j\in [m]$ or $\psi(j)=i$ for some $j\in [n]$, but not both.
For $(\varphi,\psi)\in J_{m,n}$, define

$$ a \shap_{\varphi,\psi}\, b:= c_1\cdots c_{m+n}, \text{ where }
c_i:=c_{\varphi,\psi,i}=\left\{\begin{array}{ll} a_j,& \text{if } i=\varphi(j)\text{ for } j\in [m]\\ b_j,& \text{if } i=\varphi(j) \text{ for } j\in [n] \end{array}\right ..
$$
We then have
\begin{equation}
a\shap b = \sum_{(\varphi,\psi)\in J_{m,n}} a \shap_{\varphi,\psi}\, b.
\mlabel{eq:sha}
\end{equation}

To describe the shuffle product of words of the forms $a=a_{m_1}\cdots a_{m_2}$ and $b=b_{n_1}\cdots b_{n_2}$, define the set
$$J_{[m_1,m_2],[n_1,n_2]}:=\left\{ (\varphi, \psi)\,\bigg|\, \begin{array}{l} \varphi:[m_1,m_2]\to [m_1+n_1-1,m_2+n_2]\\
\text{and } \psi:[n_1,n_2]\to [m_1+n_1-1,m_2+n_2] \text{ are order preserving}\\ \text{ and injective map such that } \im\,\varphi \cup \im\, \psi = [m_1+n_1-1,m_2+n_2]\end{array} \right\}.$$
For $(\varphi,\psi)\in J_{[m_1,m_2],[n_1,n_2]}$, define
$$ a \shap_{\varphi,\psi}\, b:= c_{m_1+n_1-1}\cdots c_{m+n}, \text{ where }
c_i:=c_{\varphi,\psi,i}=\left\{\begin{array}{ll} a_j,& \text{if } i=\varphi(j)\text{ for } j\in [m_1,m_2]\\ b_j,& \text{if } i=\varphi(j) \text{ for } j\in [n_1,n_2] \end{array}\right ..
$$
Then we have
\begin{equation}
a\shap b = \sum_{(\varphi,\psi)\in J_{[m_1,m_2],[n_1,n_2]}} a \shap_{\varphi,\psi}\, b.
\mlabel{eq:sha2}
\end{equation}

\begin{lemma}
Let $X$ be a nonempty set. Let $\calh^{\shap}$ be the shuffle product algebra on $X$ with the shuffle product $\shap$. Let $a=a_1\cdots a_m$ and $b=b_1\cdots b_n$ with $a_i, b_j\in X$, $1\leq i\leq m, 1\leq j\leq n$. For a fixed $k$ with $1\leq k\leq m$, we have
\begin{equation}
a\shap b= \sum_{i=0}^n ((a_1\cdots a_{k-1})\shap (b_1\cdots b_i))a_{k}((a_{k+1}\cdots a_m)\shap (b_{i+1}\cdots b_n)).
\mlabel{eq:sh}
\end{equation}
\mlabel{lem:pre}
\end{lemma}

\begin{proof}
For the fixed $k\in [m]$ in the lemma and for a given $i\in [n]$, denote
$$J^i_{m,n}:=J^{k,i}_{m,n}:=\{(\varphi,\psi)\in J_{m,n}\,|\, \varphi(k)=i+k\}.$$
Then we have the disjoint union
$$
J_{m,n}=\bigsqcup^{n}_{i=0}J^{i}_{m,n}\,.
$$
Further, for $(\varphi,\psi)\in J^i_{m,n}$, we have
$$ \left\{\begin{array}{l}
\varphi|_{[k-1]}:[k-1]\rightarrow[i+k-1],\\
\varphi|_{[k+1,m]}:[k+1,m]\rightarrow[i+k+1,m+n],\\
\psi|_{[i]}:[i]\rightarrow[i+k-1],\\
\psi|_{[i+1,n]}:[i+1,n]\rightarrow[i+k+1,m+n].
\end{array} \right .
$$
Then we have
$(\varphi|_{[k-1]},\psi|_{[j]})\in J_{k-1,j}=J_{[k-1],[j]}$ and $(\varphi|_{[k+1,m]},\psi|_{[i+1,n]})\in J_{[k+1,m],[i+1,n]}$.
Hence
\begin{eqnarray*}
a\shap b&=&\sum_{(\phi,\varphi)\in J_{m,n}}{a\shap_{(\varphi,\psi)} b}\\
&=&\sum_{0\leq i\leq n}{\sum_{(\phi,\varphi)\in J^{i}_{m,n}}{a\shap_{(\varphi,\psi)} b}}\\
&=&\sum_{0\leq i\leq n}{((a_1\cdots a_{k-1})\shap (b_1\cdots b_i))a_{k}((a_{k+1}\cdots a_m)\shap( b_{i+1}\cdots b_n))}.
\end{eqnarray*}
This completes the proof.
\end {proof}
By the above lemma, we easily get

\begin{theorem}
Let $k, \ell\geq 1$. For $y_1,\cdots y_k, z_1,\cdots, z_\ell\in X$ and $m_1,\cdots, m_k, n_1,\cdots, n_\ell\in \ZZ_{\geq 1}$, we have
\begin{eqnarray*}
&& (y_1^{m_1}\cdots y_k^{m_k})\shap (z_1^{n_1}\cdots z_\ell^{n_\ell})\\
&=&\sum_{1\leq j\leq \ell }\sum_{\substack{ 1\leq n_{j_1}\leq n_j
 \\n_{j_1}+n_{j_2}=n_{j}} }(y_1^{m_1-1}\shap (z_1^{n_1}\cdots z_1^{n_{j_1}}))y_1\big((y_2^{m_2}\cdots y_k^{m_k})\shap (z_1^{n_{j_2}}\cdots z_\ell^{n_\ell})\big)+y_1^{m_1} \big((y_2^{m_2}\cdots y_k^{m_k})\shap (z_1^{n_{1}}\cdots z_\ell^{n_\ell})\big).
\end{eqnarray*}
\mlabel{thm:rec}
\end{theorem}
\begin{proof}
In Lemma~\mref{lem:pre} take $k=m_1$ so that $a_k$ is the last $y_1$ (from the left) in $y_1^{m_1}=\underbrace{y_1\cdots y_1}_{m_1 \text{ factors}}$. Then the theorem follows. The last term in the formula corresponds to the term when $i=0$ in the lemma.
\end{proof}

\section{Proofs of the main theorems}
\mlabel{sec:proof}

We prove our main theorems on restricted decompositions of MZVs in this section. After a brief general discussion in Section~\mref{ss:general},
we apply Theorem~\mref{thm:rec} to prove  Theorem~\mref{thm:01.01} in Section~\mref{ss:01.01} and to prove Theorem~\mref{thm:01.0101} in Section~\mref{ss:01.0101}. We also show that Euler's decomposition formula (Corollary~\mref{col:01.01}) can be derived from Theorem~\mref{thm:01.01} in Section~\mref{ss:co:01.01}.

\subsection{Decomposition formulas of MZVs}
\mlabel{ss:general}

As is well-known, an MZV has an integral representation~\mcite{LM}
\begin{equation}
\zeta(s_1,\cdots,s_k)=
 \int_0^1 \int_0^{t_1}\cdots \int_0^{t_{|\vec{s}|-1}} \frac{dt_1}{f_1(t_1)}
 \cdots \frac{dt_{|\vec{s}|}}{f_{|\vec{s}|}(t_{|\vec{s}|})}\,.
 \mlabel{eq:intrep}
\end{equation}
Here $|\vec{s}|=s_1+\cdots +s_k$ and
$$f_j(t)=\left\{\begin{array}{ll} 1-t_j, & j= s_1,s_1+s_2,\cdots, s_1+\cdots +s_k,\\
t_j, & \text{otherwise}. \end{array} \right .
$$
The MZVs span the following $\QQ$-subspace of $\RR$
$$
\mathbf{MZV}: = \QQ \{ \zeta(s_1,\cdots,s_k)\ |\ s_i\geq 1, s_1\geq 2\} \subseteq \RR.
$$

Consider
the set $X=\{x_0,x_1\}$. The shuffle product algebra $\calh^{\shap}=\calh^{\shap}_X$ in Section~\mref{sec:rec} contains the subalgebra
$
 \calh^{\shap}_0:=x_0\calh^{\shap} x_1.
$
Since the product of nested integrals like those in Eq.~(\mref{eq:intrep}) is governed by the shuffle product, the linear map
\begin{equation}
 \zeta^{\shap}: \calh^{\shap}_0 \to \mathbf{MZV},
 \quad x_0^{  s_1-1}  x_1  \cdots   x_0^{s_k-1}  x_1 \mapsto \zeta(s_1,\cdots,s_k)
\mlabel{eq:shmap}
\end{equation}
is an algebra homomorphism~\mcite{Ho1,IKZ}. In other words,

\begin{equation} \zeta^\shap(a)\zeta^\shap(b) = \zeta^{\shap}(a\shap b) \text{ for } a, b\in \calh^{\shap}_0.
\mlabel{eq:mzvsh}
\end{equation}
This gives a so called decomposition formula when the product $a\shap b$ on the right hand side is explicitly calculated even though a decomposition formula can also mean a formula derived from a process equivalent to the shuffle product, such as the one obtained in~\mcite{EW}. For example, when $a=x_0^mx_1$ and $b=x_0^nx_1$ this gives the Euler decomposition formula. In the general case, this leads to the generalized decomposition formula in~\mcite{GX3} which gets complicated quickly as the number of variable in the MZVs increases. For $a$ and $b$ is ``restricted" forms such as $a=x_0^mx_1^j$ and $b=x_0^nx_1^k$, it is possible to find simpler formulas for $a\shap b$, giving restricted decomposition formulas.

In this paper, we prove Theorem~\mref{thm:01.01} and Theorem~\mref{thm:01.0101} by applying Theorem~\mref{thm:rec} to recursively calculate shuffle product formulas for $x_0^mx_1^j\shap x_0^nx_1^k$ and $x_0^mx_1^j\shap x_0^nx_1^kx_0^sx_1^t$ respectively. We then apply Eq.~(\mref{eq:mzvsh}) to obtain equations of MZVs in the theorems.

\subsection{The proof of Theorem~\mref{thm:01.01}}
\mlabel{ss:01.01}
To prove Theorem~\mref{thm:01.01}, we first recall the following well-known shuffle product formulas
\begin{equation}
x^m_0\shap x^n_0=\bincc{m+n}{m}x^{m+n}_0
\mlabel{eq:sh0.0}
\end{equation}
and
\begin{equation}
x^m_1\shap x^n_1=\bincc{m+n}{m}x^{m+n}_1.
\mlabel{eq:sh1.1}
\end{equation}
We further have
\begin{equation}
x^m_0\shap x^n_1=\sum\limits_{ m_1+m_2+\cdots m_{n+1}=m,
m_i\geq 0}x^{m_1}_0x_1x^{m_2}_0x_1\cdots x^{m_{n}}_0x_1 x^{m_{n+1}}_0.
\mlabel{eq:sh0.1}
\end{equation}
where we denote the sum by $\frakB^{m}_{n+1}$. Thus we also have
\begin{equation}
x^m_1\shap x^n_0=\sum\limits_{ n_1+n_2+\cdots n_{m+1}=n,
n_i\geq 0}x^{n_1}_0x_1x^{n_2}_0x_1\cdots x^{n_{m}}_0x_1 x^{n_{m+1}}_0=\frakB^{n}_{m+1}.
\mlabel{eq:sh1.0}
\end{equation}
By Theorem~\mref{thm:rec}, Eqs.~(\mref{eq:sh1.1}) and (\mref{eq:sh1.0}), we obtain
\allowdisplaybreaks{\begin{eqnarray}
x^m_1\shap x^n_1x^k_0&=&\sum\limits_{ m_1+m_2=m,
m_i\geq 0}(x^{m_1}_1\shap x^{n-1}_1)x_1(x^{m_2}_1\shap  x^{k}_0)\mlabel{eq:sh1.10}
\\
&=&\sum\limits_{ m_1+m_2=m,
m_i\geq 0}\bincc{m_1+n-1}{n-1}x^{m_1+n}_1\frakB^{k}_{m_2+1}.\notag
\end{eqnarray}}
Likewise we have

\allowdisplaybreaks{\begin{eqnarray}
x^m_0\shap x^n_1x^k_0&=&\sum\limits_{ m_1+m_2=m,
m_i\geq 0}(x^{m_1}_0\shap x^{n}_1)x_0(x^{m_2}_0\shap  x^{k-1}_0)\mlabel{eq:sh0.10}
\\
&=&\sum\limits_{ m_1+m_2=m,
m_i\geq 0}\bincc{m_2+k-1}{k-1}\frakB^{m_1}_{n+1}x^{m_2+k}_0,\notag
\end{eqnarray}}

\allowdisplaybreaks{\begin{eqnarray}
x^m_1\shap x^n_0x^k_1&=&\sum\limits_{ m_1+m_2=m,
m_i\geq 0}(x^{m_1}_1\shap x^{n}_0)x_1(x^{m_2}_1\shap  x^{k-1}_1)\mlabel{eq:sh1.01}
\\
&=&\sum\limits_{ m_1+m_2=m,
m_i\geq 0}\bincc{m_2+k-1}{k-1}\frakB^{n}_{m_1+1}x^{m_2+k}_1\notag
\end{eqnarray}}
and

\allowdisplaybreaks{\begin{eqnarray}
x^m_0\shap x^n_0x^k_1&=&\sum\limits_{ m_1+m_2=m,
m_i\geq 0}(x^{m_1}_0\shap x^{n-1}_0)x_0(x^{m_2}_0\shap  x^{k}_1)\mlabel{eq:sh0.01}
\\
&=&\sum\limits_{ m_1+m_2=m,
m_i\geq 0}\bincc{m_1+n-1}{n-1}x^{m_1+n}_0\frakB^{m_2}_{k+1}.\notag
\end{eqnarray}}

By Theorem~\mref{thm:rec}, Eqs.~(\mref{eq:sh1.01}), (\mref{eq:sh0.1}) and (\mref{eq:sh1.0}), we obtain
\allowdisplaybreaks{\begin{eqnarray}
x^m_1\shap x^n_1x^k_0x^s_1
&=& \sum_{m_1+m_1'=m,m_1,m_1'\geq 0} (x_1^{m_1}\shap x_1^{n-1})x_1\big(x_1^{m_1'}\shap (x_0^kx_1^s)\big) \mlabel{eq:sh1.101}\\
&=& \sum_{m_1+m_1'=m,m_1,m_1'\geq 0} (x_1^{m_1}\shap x_1^{n-1})x_1\left(\sum_{m_2+m_3=m_1', m_2, m_3\geq 0} (x_1^{m_2}\shap x_0^{k})x_1 (x_1^{m_3}\shap x_1^{s-1})\right)\notag\\
&=&\sum\limits_{ m_1+m_2+m_3=m,
m_i\geq 0}(x^{m_1}_1\shap x^{n-1}_1)x_1(x^{m_2}_1\shap  x_0^{k})x_1(x^{m_3}_1\shap x^{s-1}_1)\notag\\
&=&\sum\limits_{ m_1+m_2+m_3=m,
m_i\geq 0}\bincc{m_1+n-1}{n-1}\bincc{m_3+s-1}{s-1}x^{m_1+n}_1 \frakB^{k}_{m_2+1}x^{m_3+s}_1.\notag
\end{eqnarray}}

By a similar argument, we obtain the following Eqs.~(\mref{eq:sh0.010}) -- (\mref{eq:sh0.101}).
\allowdisplaybreaks{\begin{eqnarray}
x^m_0\shap x^n_0x^k_1x^s_0&=&\sum\limits_{ m_1+m_2+m_3=m,
m_i\geq 0}(x^{m_1}_0\shap x^{n-1}_0)x_0(x^{m_2}_0\shap  x^{k}_1)x_0(x^{m_3}_0\shap x^{s-1}_0)\mlabel{eq:sh0.010}
\\
&=&\sum\limits_{ m_1+m_2+m_3=m,
m_i\geq 0}\bincc{m_1+n-1}{n-1}\bincc{m_3+s-1}{s-1} x^{m_1+n}_0\frakB^{m_2}_{k+1}x^{m_3+s}_0.\notag
\end{eqnarray}}

\allowdisplaybreaks{\begin{eqnarray}
x^m_1\shap x^n_0x^k_1x^s_0&=&\sum\limits_{ m_1+m_2+m_3=m,
m_i\geq 0}(x^{m_1}_1\shap x^{n-1}_0)x_0(x^{m_2}_1\shap  x^{k-1}_1)x_1(x^{m_3}_1\shap x^{s}_0)\mlabel{eq:sh1.010}\\
&=&\sum\limits_{ m_1+m_2+m_3=m,
m_i\geq 0}\bincc{m_2+k-1}{k-1}\frakB^{n-1}_{m_1+1} x_0x^{m_2+k}_1\frakB^{s}_{m_3+1}.\notag
\end{eqnarray}}

\allowdisplaybreaks{\begin{eqnarray}
x^m_0\shap x^n_1x^k_0x^s_1&=&\sum\limits_{ m_1+m_2+m_3=m,
m_i\geq 0}(x^{m_1}_0\shap x^{n-1}_1)x_1(x^{m_2}_0\shap  x^{k-1}_0)x_0(x^{m_3}_0\shap x^{s}_1)\mlabel{eq:sh0.101}
\\
&=&\sum\limits_{ m_1+m_2+m_3=m, m_i\geq 0}\bincc{m_2+k-1}{k-1}\frakB^{m_1}_{n}x_1x^{m_2+k}_0\frakB^{m_3}_{s+1}.\notag
\end{eqnarray}}

By Theorem~\mref{thm:rec} and Eq.~(\mref{eq:sh0.101}), we further obtain
\allowdisplaybreaks{\begin{eqnarray}
&&x^m_0\shap x^n_0x^k_1x^s_0x^t_1\mlabel{eq:sh0.0101}
\\
=&&\sum\limits_{ m_1+m_2+m_3+m_4=m,
m_i\geq 0}(x^{m_1}_0\shap x^{n-1}_0)x_0(x^{m_2}_0\shap  x^{k-1}_1)x_1(x^{m_3}_0\shap x^{s-1}_0)x_0(x^{m_4}_0\shap  x^{t}_1)\notag\\
=&&\sum\limits_{ m_1+m_2+m_3+m_4=m,
m_i\geq 0}\bincc{m_1+n-1}{n-1}\bincc{m_3+s-1}{s-1}x^{m_1+n}_0\frakB^{m_2}_{k}x_1x^{m_3+s}_0\frakB^{m_4}_{t+1}.\notag
\end{eqnarray}}
We similarly have
\allowdisplaybreaks{\begin{eqnarray}
&&x^m_1\shap x^n_0x^k_1x^s_0x^t_1\mlabel{eq:sh1.0101}
\\
=&&\sum\limits_{ m_1+m_2+m_3+m_4=m,
m_i\geq 0}(x^{m_1}_1\shap x^{n}_0)x_1(x^{m_2}_1\shap  x^{k-1}_1)x_0(x^{m_3}_1\shap x^{s-1}_0)x_1(x^{m_4}_1\shap  x^{t-1}_1)\notag\\
=&&\sum\limits_{ m_1+m_2+m_3+m_4=m,
m_i\geq 0}\bincc{m_2+k-1}{k-1}\bincc{m_4+t-1}{t-1}\frakB^{n}_{m_1+1}x^{m_2+k}_1x_0\frakB^{s-1}_{m_3+1}x^{m_4+t}_1.\notag
\end{eqnarray}}

Using Theorem~\mref{thm:rec} and then Eqs.~(\mref{eq:sh0.0}), (\mref{eq:sh1.01}), (\mref{eq:sh0.01}) and (\mref{eq:sh1.1}), we derive

\allowdisplaybreaks{\begin{eqnarray}
&&x^m_0x^j_1\shap x^n_0x^k_1\mlabel{eq:sh01.01}
\\
&=&\sum\limits_{ 0\leq n_1\leq n}(x^{m-1}_0\shap x^{n_1}_0)x_0(x^{j}_1\shap x^{n-n_1}_0  x^{k}_1)+\sum\limits_{ 1\leq k_1\leq k}(x^{m-1}_0\shap x^{n}_0x^{k_1}_1)x_0(x^j_1\shap  x^{k-k_1}_1)\notag\\
&=&\sum\limits_{ 0\leq n_1\leq n}\bincc{m-1+n_1}{m-1}x^{m+n_1}_0\sum\limits_{ j_1+j_2=j,
j_i\geq 0}\bincc{j_2+k-1}{k-1}\frakB^{n-n_1}_{j_1+1}x^{j_2+k}_1\notag\\
&&+\sum\limits_{ 1\leq k_1\leq k}\sum\limits_{m_1+m_2=m-1,
m_i\geq 0}\bincc{m_1+n-1}{n-1}x^{m_1+n}_0 \frakB^{m_2}_{k_1+1}\bincc{j+k-k_1}{j}x_0x^{j+k-k_1}_1\notag
\\
&=&\sum\limits_{ 0\leq n_1\leq n}\bincc{m-1+n_1}{m-1}\sum\limits_{ j_1+j_2=j,
j_i\geq 0}\bincc{j_2+k-1}{k-1}\sum\limits_{ \alpha_1+\cdots +\alpha_{j_1+1}=n-n_1,
\alpha_i\geq 0}x_0^{\alpha_1+m+n_1}x_1x_0^{\alpha_2}x_1\cdots x_0^{\alpha_{j_1+1}}x_1^{j_2+k}\notag\\
&&+\sum\limits_{ 1\leq k_1\leq k}\sum\limits_{ m_1+m_2=m-1,
m_i\geq 0}\bincc{m_1+n-1}{n-1}\sum\limits_{ \beta_1+\cdots +\beta_{k_1+1}=m_2,
\beta_i\geq 0}\bincc{j+k-k_1}{j} x_0^{\beta_1+m_1+n}x_1x_0^{\beta_2}\cdots x_0^{\beta_{k_1}}x_1 x_0^{\beta_{k_1+1}+1}x_1^{j+k-k_1}.\notag
\end{eqnarray}}
Applying the algebra homomorphism $\zeta^\shap$ in Eq.~(\mref{eq:shmap}) to both sides of the above equation. By Eq.~(\mref{eq:mzvsh}), the left hand side becomes $\zeta(m+1,\{1\}^{j-1})\zeta(n+1,\{1\}^{k-1})$. The right hand side becomes the right hand side of the equation in Theorem~\mref{thm:01.01}. This proves Theorem~\mref{thm:01.01}.

\subsection{The proof of Corollary~\mref{col:01.01}}
\mlabel{ss:co:01.01}

We now derive Euler's decomposition formula, namely Corollary~\mref{col:01.01}, from Theorem~\mref{thm:01.01}. We recall the formula
\begin{equation}
\sum_{s=0}^k \bincc{m-1+s}{m-1}=\bincc{m+k}{m} \text{ or }
\sum_{t=k}^n \bincc{m-1+n-t}{m-1}=\bincc{m+n-k}{m}, \quad m, k\geq 0, 0\leq k\leq n,
\mlabel{eq:pas}
\end{equation}
which can be proved from the Pascal's rule by an induction.

Taking $j=k=1$ in Theorem~\mref{thm:01.01} and writing the case when $j_1=0$ in the first sum separately, we obtain
\allowdisplaybreaks{\begin{eqnarray}
&&\zeta(m+1)\zeta(n+1)\mlabel{eq:sh11}\\
&=&\sum_{
\substack{ 0\leq n_1 \leq n}} \bincc{m-1+n_1}{m-1}\zeta(m+n+1,1)+\sum_{
\substack{ 0\leq n_1 \leq n,|\alpha|=n-n_1+2},\alpha_i\geq 1} \bincc{m-1+n_1}{m-1}\zeta(\alpha_1+m+n_1,\alpha_2)\notag
\\
&&+\sum_{
\substack{|\beta|=m_2+2,m_1+m_2=m-1,m_i\geq 0
}} \bincc{m_1+n-1}{n-1}\zeta(\beta_1+m_1+n,\beta_2+1).\notag
\end{eqnarray}}
The second sum in Eq.~(\mref{eq:sh11}) equals
\allowdisplaybreaks{\begin{eqnarray}
&&\sum_{
\substack{ 0\leq n_1 \leq n}} \bincc{m-1+n_1}{m-1}\sum_{
\substack{ |\alpha|=n-n_1+2,\alpha_i\geq 1}} \zeta(\alpha_1+m+n_1,\alpha_2)\notag\\
&=&\sum_{n_1=0}^{n}
\bincc{m+n_1-1}{m-1}\sum_{\alpha_2=1}^{
n-n_1+1}\zeta(m+n+2-\alpha_2,\alpha_2)\notag
\end{eqnarray}}
Taking $t:=n-n_1$, exchanging the order of the summations and applying Eq.~(\mref{eq:pas}), we obtain
\allowdisplaybreaks{\begin{eqnarray}
&&\sum_{t=1}^{n}
\bincc{m+n-t-1}{m-1}\sum_{\alpha_2=1}^{
t+1}\zeta(m+n+2-\alpha_2,\alpha_2)\notag\\
&=&\sum_{\alpha_2=1}^{n+1}
\sum_{t=\alpha_2-1}^{n}\bincc{m+n-t-1}{m-1}\zeta(m+n+2-\alpha_2,\alpha_2)\notag\\
&=&\sum_{\alpha_2=1}^{n+1}\bincc{m+n-\alpha_2+1}{m}\zeta(m+n+2-\alpha_2,\alpha_2).\notag
\end{eqnarray}}
This is the first sum in Eq.~(\mref{eq:co:01.01}).

Applying the same argument to the third sum in Eq.~(\mref{eq:sh11}), we derive
\allowdisplaybreaks{\begin{eqnarray}
&&\sum_{
\substack{ m_1+m_2=m-1,m_i\geq 0}} \bincc{m_1+n-1}{n-1}\sum_{\substack{
|\beta|=m_2+2,\beta_i\geq 1}}\zeta(\beta_1+m_1+n,\beta_2+1)\notag\\
&=&\sum_{m_1=0}^{m-1}
\bincc{m_1+n-1}{n-1}\sum_{\beta_2=1}^{
m-m_1}\zeta(m+n+1-\beta_2,\beta_2+1)\notag\\
&=&\sum_{k=1}^{m}
\bincc{m+n-k-1}{n-1}\sum_{\beta_2=2}^{
k+1}\zeta(m+n+2-\beta_2,\beta_2)\notag\\
&=&\sum_{\beta_2=2}^{
m+1}\sum_{k=\beta_2-1}^{m}
\bincc{m+n-k-1}{n-1}\zeta(m+n+2-\beta_2,\beta_2)\notag\\
&=&\sum_{\beta_2=2}^{
m+1}
\bincc{m+n-\beta_2+1}{n}\zeta(m+n+2-\beta_2,\beta_2).\notag
\end{eqnarray}}

By Eq.~(\mref{eq:pas}), the first sum in Eq.~(\mref{eq:sh11}) is
$\bincc{m+n}{m}\zeta(m+n+1,1)$. Combining it with the above sum, we see that the first sum and the third sum in Eq.~(\mref{eq:sh11}) give

$$\sum_{\beta_2=1}^{
m+1}
\bincc{m+n-\beta_2+1}{n}\zeta(m+n+2-\beta_2,\beta_2).
$$
This is the second sum in Eq.~(\mref{eq:co:01.01}). Thus Corollary~\mref{col:01.01} is proved.

\subsection{The proof of Theorem~\mref{thm:01.0101}}
\mlabel{ss:01.0101}
We finally prove Theorem~\mref{thm:01.0101}.

Applying Theorem~\mref{thm:rec} and then Eqs.~(\mref{eq:sh0.0}), (\mref{eq:sh1.0101}), (\mref{eq:sh0.01}), (\mref{eq:sh1.101}), (\mref{eq:sh0.010}), (\mref{eq:sh1.01}), (\mref{eq:sh0.0101}) and (\mref{eq:sh1.1}), we obtain
\allowdisplaybreaks{\begin{eqnarray}
&&x^m_0x^j_1\shap x^n_0x^k_1x^s_0x^t_1 \notag
\\
=&&\sum\limits_{ 0\leq n_1\leq n}(x^{m-1}_0\shap x^{n_1}_0)x_0(x^{j}_1\shap x^{n-n_1}_0  x^{k}_1x^s_0x^t_1)+\sum\limits_{ 1\leq k_1\leq k}(x^{m-1}_0\shap x^{n}_0x^{k_1}_1)x_0(x^j_1\shap  x^{k-k_1}_1x^s_0x^t_1)\notag\\
&+&\sum\limits_{ 1\leq s_1\leq s}(x^{m-1}_0\shap x^{n}_0x^{k}_1x^{s_1}_0)x_0(x^j_1\shap  x^{s-s_1}_0x^t_1)+\sum\limits_{ 1\leq t_1\leq t}(x^{m-1}_0\shap x^{n}_0x^{k}_1x^s_0x^{t_1}_1)x_0(x^j_1\shap  x^{t-t_1}_1)\notag\\
=&&\sum\limits_{ 0\leq n_1\leq n}\bincc{m-1+n_1}{m-1}x^{m+n_1}_0\sum\limits_{ j_1+j_2+j_3+j_4=j,
j_i\geq 0}\bincc{j_2+k-1}{k-1}\bincc{j_4+t-1}{t-1}\frakB^{n-n_1}_{j_1+1}x^{j_2+k}_1x_0\frakB^{s-1}_{j_3+1}x^{j_4+t}_1\notag\\
&+&\sum\limits_{ 1\leq k_1\leq k}\sum\limits_{ m_1+m_2=m-1,
m_i\geq 0}\bincc{m_1+n-1}{n-1}x^{m_1+n}_0\frakB^{m_2}_{k_1+1}x_0\sum\limits_{ j_1+j_2+j_3=j,
j_i\geq 0}\bincc{j_1+k-k_1}{k-k_1}\bincc{j_3+t-1}{t-1}x^{j_1+k-k_1}_1x_0\frakB^{s-1}_{j_2+1}x^{j_3+t}_1\notag\\
&+&\sum\limits_{ 1\leq s_1\leq s}\sum\limits_{ m_1+m_2+m_3=m-1,
m_i\geq 0}\bincc{m_1+n-1}{n-1}\bincc{m_3+s_1-1}{s_1-1}x^{m_1+n}_0\frakB^{m_2}_{k+1}x^{m_3+s_1+1}_0\sum\limits_{j_1+j_2=j,
j_i\geq 0}\bincc{j_2+t-1}{t-1}\frakB^{s-s_1}_{j_1+1}x^{j_2+t}_1\notag\\
&+&\sum\limits_{ 1\leq t_1\leq t}\sum\limits_{ m_1+m_2+m_3+m_4=m-1,
m_i\geq 0 }\bincc{m_1+n-1}{n-1}\bincc{m_3+s-1}{s-1}\bincc{j+t-t_1}{j}x^{m_1+n}_0\frakB^{m_2}_{k}x_1x^{m_3+s}_0\frakB^{m_4}_{t_1+1}x_0x^{j+t-t_1}_{1}\notag
\\
=&&\sum\limits_{ 0\leq n_1\leq n}\bincc{m-1+n_1}{m-1}\sum\limits_{ j_1+j_2+j_3+j_4=j,
j_i\geq 0}\bincc{j_2+k-1}{k-1}\bincc{j_4+t-1}{t-1}\sum\limits_{ \alpha_1+\cdots +\alpha_{j_1+1}=n-n_1,
\alpha_i\geq 0}\sum\limits_{\tilde{\alpha}_1+\cdots +\tilde{\alpha}_{j_3+1}=s-1,
\tilde{\alpha}_i\geq 0}\notag\\
&&x_0^{\alpha_1+m+n_1}x_1x_0^{\alpha_2}x_1\cdots x_0^{\alpha_{j_1}}x_1x_0^{\alpha_{j_1+1}}x_1^{j_2+k} x_0^{\tilde{\alpha}_1+1}x_1\cdots x_0^{\tilde{\alpha}_{j_3}}x_1 x_0^{\tilde{\alpha}_{j_3+1}}x_1^{j_4+t}\notag\\
&+&\sum\limits_{ 1\leq k_1\leq k}\sum\limits_{ m_1+m_2=m-1,
m_i\geq 0}\bincc{m_1+n-1}{n-1}\sum\limits_{ \beta_1+\cdots +\beta_{{k_1}+1}=m_2,
\beta_i\geq 0}\sum\limits_{ j_1+j_2+j_3=j,
j_i\geq 0}\bincc{j_1+k-k_1}{k-k_1}\bincc{j_3+t-1}{t-1} \sum\limits_{\tilde{\beta}_1+\cdots +\tilde{\beta}_{j_2+1}=s-1,
\tilde{\beta}_i\geq 0}\notag\\
&&x_0^{\beta_1+m_1+n}x_1x_0^{\beta_2}x_1\cdots x_0^{\beta_{k_1}}x_1 x_0^{\beta_{k_1+1}+1}x_1^{j_1+k-k_1}x_0^{\tilde{\beta}_1+1}x_1\cdots x_0^{\tilde{\beta}_{j_2}}x_1x_0^{\tilde{\beta}_{j_2+1}}x_1^{j_3+t}\notag\\
&+&\sum\limits_{ 1\leq s_1\leq s}\sum\limits_{ m_1+m_2+m_3=m-1,
m_i\geq 0}\bincc{m_1+n-1}{n-1}\bincc{m_3+s_1-1}{s_1-1}
\sum\limits_{ \gamma_1+\cdots +\gamma_{k+1}=m_2,
\gamma_i\geq 0}\sum\limits_{j_1+j_2=j,
j_i\geq 0}\bincc{j_2+t-1}{t-1}\sum\limits_{\tilde{\gamma}_1+\cdots +\tilde{\gamma}_{j_1+1}=s-s_1,
\tilde{\gamma}\geq 0}\notag\\
&&x_0^{\gamma_1+m_1+n}x_1x_0^{\gamma_2}x_1 \cdots x_0^{\gamma_{k}}x_1 x_0^{\gamma_{k+1}+\tilde{\gamma}_1+m_3+s_1+1}x_1 x_0^{\tilde{\gamma}_2}x_1\cdots x_0^{\tilde{\gamma}_{j_1}}x_1 x_0^{\tilde{\gamma}_{j_1+1}}x_1^{j_2+t}\notag\\
&+&\sum\limits_{ 1\leq t_1\leq t}\sum\limits_{ m_1+m_2+m_3+m_4=m-1,
m_i\geq 0}\bincc{m_1+n-1}{n-1}\bincc{m_3+s-1}{s-1}\bincc{j+t-t_1}{j}\sum\limits_{ \delta_1+\cdots +\delta_{k}=m_2,
\delta_i\geq 0}\sum\limits_{\tilde{\delta}_1+\cdots +\tilde{\delta}_{t_1+1}=m_4,
\tilde{\delta}_i\geq 0}\notag\\
&&x_0^{\delta_1+m_1+n}x_1x_0^{\delta_2}x_1\cdots x_0^{\delta_{k}}x_1x_0^{m_3+s}x_0^{\tilde{\delta}_1}x_1 x_0^{\tilde{\delta}_2}x_1\cdots x_0^{\tilde{\delta}_{t_1}}x_1x_0^{\tilde{\delta}_{t_1+1}+1}x_1^{j+t-t_1}.\notag \end{eqnarray}}
Applying the algebra homomorphism $\zeta^\shap$ in Eq.~(\mref{eq:shmap}) to both sides of the above equation. By Eq.~(\mref{eq:mzvsh}), the left hand side becomes $\zeta(m+1,\{1\}^{j-1})\zeta(n+1,\{1\}^{k-1},s+1,\{1\}^{t-1})$. The right hand side becomes the right hand side in the equation in Theorem~\mref{thm:01.0101}. This proves Theorem~\mref{thm:01.0101}.
\smallskip

\begin{remark}
As noted in the introduction, our method in this paper in principle can be applied to derive a restricted decomposition formula for a product of two MZVs with any numbers of strings of 1's. But the formula becomes complicated quickly. For example, to derive the formula for a product of two MZVs both with two strings of 1's, we computed $x_0^mx_1^jx_0^rx_1^\ell\shap x_0^nx_1^kx_0^sx_1^t$ and obtained 20 nested sums. So we do not present the formula here. 
 \mlabel{rk:0101.0101}
\end{remark}
\smallskip

\noindent {\bf Acknowledgements}: This work is supported by the National Natural Science Foundation of China (Grant No. 11371178) and the National Science Foundation of US (Grant No. DMS~1001855).

\end{document}